\documentclass[a4paper,12pt]{amsart}
\usepackage[margin=2.5cm]{geometry}
\usepackage[utf8]{inputenc}
\usepackage[T1]{fontenc}
\usepackage[charter]{mathdesign}
\usepackage{hyperref}
\hypersetup{
bookmarks=true,
pdfpagemode=UseNone,
pdfstartview=FitH,
pdfdisplaydoctitle=true,
pdfborder={0 0 0}, 
pdftitle={Return times theorem for amenable groups},
pdfauthor={Pavel Zorin-Kranich},
pdfsubject={Mathematics Subject Classification (2010): 37A15},
pdfkeywords={return times theorem, Wiener-Wintner theorem},
pdflang=en-US
}
\usepackage[safeinputenc=true,style=alphabetic,firstinits,maxnames=4,useprefix=true]{biblatex}

\numberwithin{equation}{section}
\newtheorem{theorem}[equation]{Theorem}
\newtheorem{lemma}[equation]{Lemma}

\newtheorem{corollary}[equation]{Corollary}
\theoremstyle{definition}
\newtheorem{definition}[equation]{Definition}
\theoremstyle{remark}

\newcommand*{\N}{\mathbb{N}}
\newcommand*{\Z}{\mathbb{Z}}

\newcommand*{\R}{\mathbb{R}}

\newcommand*{\dif}{\mathrm{d}}

\newcommand*{\inv}{^{-1}}
\newcommand*{\E}{\mathbb{E}}

\newcommand*{\ul}[1]{\underline{#1}}
\newcommand{\ud}{\overline{d}}
\newcommand{\ld}{\underline{d}}

\begin{document}
\title{Return times theorem for amenable groups}
\author{Pavel Zorin-Kranich}
\address
{Institute of Mathematics\\
Hebrew University, Givat Ram\\
Jerusalem, 91904, Israel}
\email{pzorin@math.huji.ac.il}
\urladdr{http://math.huji.ac.il/~pzorin/}
\begin{abstract}
We extend Bourgain's return times theorem to arbitrary locally compact second countable amenable groups.
The proof is based on a version of the Bourgain--Furstenberg--Katznelson--Ornstein orthogonality criterion.
\end{abstract}
\maketitle
\allowdisplaybreaks

\section{Introduction}
Bourgain's return times theorem states that for every ergodic measure-preserving system $(X,T)$ and $f\in L^{\infty}(X)$, for a.e.\ $x\in X$ the sequence $c_{n}=f(T^{n}x)$ is a \emph{good sequence of weights} for the pointwise ergodic theorem, i.e.\ for every measure-preserving system $(Y,S)$ and $g\in L^{\infty}(Y)$, for a.e. $y\in Y$ the averages
\[
\frac1N \sum_{n=1}^{N} c_{n} g(S^{n}y)
\]
converge as $N\to\infty$.
This has been extended to discrete amenable groups for which an analog of the Vitali covering lemma holds by Ornstein and Weiss \cite{MR1195256}.

In this article we use the Lindenstrauss random covering lemma to extend this result to not necessarily discrete locally compact second countable (lcsc) amenable groups.
It has been observed by Lindenstrauss that this is possible in the discrete case.
In the non-discrete case we have to restrict ourselves to the class of \emph{strong} F\o{}lner sequences (see Definition~\ref{def:folner}), but we will show that every lcsc amenable group admits such a sequence.

A secondary motivation is to formulate and prove the Bourgain--Furstenberg--Katznelson--Ornstein (BFKO) orthogonality criterion \cite{MR1557098} at an appropriate level of generality.
This criterion provides a sufficient condition for the values of a function along an orbit of an ergodic measure-preserving transformation to be good weights for convergence to zero in the pointwise ergodic theorem.

The original formulation of the orthogonality criterion is slightly artificial since it assumes something about the whole measure-preserving system but concludes something that only involves a single orbit.
A more conceptual approach is to find a condition that identifies good weights and to prove that it is satisfied along almost all orbits of a measure-preserving system in a separate step.
For $\Z$-actions this seems to have been first explicitly mentioned in \cite[\textsection 4]{MR1286798}.
In order to state the appropriate condition for general lcsc amenable groups we need some notation.

Throughout the article $G$ denotes a lcsc amenable group with left Haar measure $|\cdot|$ and $(F_{N})$ a F\o{}lner sequence in $G$ that is usually fixed.
The \emph{lower density} of a subset $S\subset G$ is defined by $\ld(S) := \liminf_{N} |S\cap F_{N}| / |F_{N}|$ and the \emph{upper density} is defined accordingly as $\ud(S) := \limsup_{N} |S\cap F_{N}| / |F_{N}|$.
All functions on $G$ that we consider are real-valued and bounded by $1$.
We denote averages by $\E_{g\in F_{n}}:=\frac{1}{|F_{n}|} \int_{g\in F_{n}}$.
For $c\in L^{\infty}(G)$ we let
\[
S_{\delta,L,R}(c) :=
\{a : \forall L\leq n\leq R\,
| \E_{g\in F_{n}}c(g)c(ga) | < \delta\}.
\]
Our orthogonality  condition on the map $c$ is then the following.
\begin{equation}
\tag{$\perp$}
\label{eq:cond}
\forall\delta>0\,
\exists N_{\delta}\in\N\,
\forall N_{\delta} \leq L \leq R\quad
\ld(S_{\delta,L,R}(c))
> 1-\delta.
\end{equation}
A very rough approximation to this condition is that there is little correlation between $c$ and its translates.
The significance of this condition is explained by the following statements.
\begin{lemma}
\label{lem:ae-perp}
Let $(X,\mu,G)$ be an ergodic measure-preserving system and $f\in L^{\infty}(X)$ be orthogonal to the Kronecker factor.
Then for a.e.\ $x\in X$ the map $g\mapsto f(gx)$ satisfies \eqref{eq:cond}.
\end{lemma}
\begin{theorem}
\label{thm:good-weight-ptw-conv-to-zero}
Assume that $(F_{N})$ is a tempered strong F\o{}lner sequence and $c\in L^{\infty}(G)$ satisfies the condition \eqref{eq:cond}.
Then for every ergodic measure-preserving system $(X,G)$ and $f\in L^{\infty}(X)$ we have
\[
\lim_{N\to\infty} \E_{g\in F_{N}} c(g) f(gx) = 0
\quad \text{for a.e.\ } x\in X.
\]
\end{theorem}
This, together with a Wiener--Wintner type result, leads to the following return times theorem.
\begin{theorem}
\label{thm:rtt-amenable}
Let $G$ be a lcsc group with a tempered strong F\o{}lner sequence $(F_{n})$.
Then for every ergodic measure-preserving system $(X,G)$ and every $f\in L^{\infty}(X)$ there exists a full measure set $\tilde X\subset X$ such that for every $x\in\tilde X$ the map $g\mapsto f(gx)$ is a good weight for the pointwise ergodic theorem along $(F_{n})$.
\end{theorem}

\section{Background and tools}
\subsection{F\o{}lner sequences}
A lcsc group is called \emph{amenable} if it admits a weak F\o{}lner sequence in the sense of the following definition.
\begin{definition}
\label{def:folner}
Let $G$ be a lcsc group with left Haar measure $|\cdot|$.
A sequence of nonnull compact sets $(F_{n})$ is called
\begin{enumerate}
\item a \emph{weak F\o{}lner sequence} if for every compact set $K\subset G$ one has $|F_{n} \Delta KF_{n}| / |F_{n}| \to 0$ and
\item a \emph{strong F\o{}lner sequence} if for every compact set $K\subset G$ one has $|\partial_{K}(F_{n})|/|F_{n}| \to 0$, where $\partial_{K}(F)= K\inv F \cap K\inv F^{\complement}$ is the \emph{$K$-boundary} of $F$.
\item \emph{($C$-)tempered} if there exists a constant $C$ such that
\[
|\cup_{i<j} F_{i}\inv F_{j}| < C |F_{j}|
\quad\text{for every }j.
\]
\end{enumerate}
\end{definition}
Every strong F\o{}lner sequence is also a weak F\o{}lner sequence.
In countable groups the converse is also true, but already in $\R$ this is no longer the case: let for example $(F_{n})$ be a sequence of nowhere dense sets $F_{n}\subset [0,n]$ of Lebesgue measure $n-1/n$, say.
This is a weak but not a strong F\o{}lner sequence (in fact, $\partial_{[-1,0]} F_{n}$ is basically $[0,n+1]$).
However, a weak F\o{}lner sequence can be used to construct a strong F\o{}lner sequence.
\begin{lemma}
Assume that $G$ admits a weak F\o{}lner sequence.
Then $G$ also admits a strong F\o{}lner sequence.
\end{lemma}
\begin{proof}
By the argument in \cite[Lemma 2.6]{2012arXiv1205.3649P} it suffices, given $\epsilon>0$ and a compact set $K\subset G$, to find a compact set $F$ with $|\partial_{K}(F)|/|F|<\epsilon$.
Let $(F_{n})$ be a weak F\o{}lner sequence, then there exists $n$ such that $|K\inv K F_{n} \Delta F_{n}| < \epsilon |F_{n}|$.
Set $F=KF_{n}$, then
\[
\partial_{K}F
= K\inv KF_{n} \cap K\inv (KF_{n})^{\complement}
\subset K\inv KF_{n}\cap F_{n}^{\complement},
\]
and this has measure less than $\epsilon |F_{n}| \leq \epsilon |F|$.
\end{proof}
Since every weak (hence also every strong) F\o{}lner sequence has a tempered subsequence \cite[Proposition 1.4]{MR1865397} this implies that every lcsc amenable group admits a tempered strong F\o{}lner sequence.

\subsection{Fully generic points}
Let $(X,G)$ be an ergodic measure-preserving system and $f\in L^{\infty}(X)$.
Recall that a point $x\in X$ is called \emph{generic} for $f$ if
\[
\lim_{n} \E_{g\in F_{n}} f(gx) = \int_{X}f.
\]
In the context of countable group actions fully generic points for $f$ are usually defined as points that are generic for every function in the closed $G$-invariant algebra spanned by $f$.
For uncountable groups this is not a good definition since this algebra need not be separable.
The natural substitute for shifts of a function $f\in L^{\infty}(X)$ is provided by convolutions
\[
c*f(x) = \int_{G} c(g) f(g\inv x) \dif g,
\quad
c\in L^{1}(G).
\]
Since $L^{1}(G)$ is separable and convolution is continuous as an operator $L^{1}(G)\times L^{\infty}(X) \to L^{\infty}(X)$, the closed convolution-invariant algebra generated by $f$ is separable.

We call a point $x\in X$ \emph{fully generic} for $f$ if it is generic for every function in this algebra.
In view of the Lindenstrauss pointwise ergodic theorem \cite[Theorem 1.2]{MR1865397}, if $(F_{n})$ is tempered, then for every $f\in L^{\infty}(X)$ a.e.\ $x\in X$ is fully generic.

Now we verify that the BFKO condition implies \eqref{eq:cond}.
\begin{lemma}
\label{lem:to-zero-in-density}
Let $G$ be a lcsc group with a tempered F\o{}lner sequence $(F_{n})$.
Let $(X,G)$ be an ergodic measure-preserving system and $f\in L^{\infty}(X)$ be bounded by $1$.
Let $x\in X$ be a fully generic point for $f$ such that
\begin{equation}
\label{eq:BFKO-cond}
\lim_{n}\E_{g\in F_{n}} f(gx)f(g\xi) = 0
\quad\text{for a.e. }\xi\in X.
\end{equation}
Then the map $g\mapsto f(gx)$ satisfies \eqref{eq:cond}.
\end{lemma}
\begin{proof}
Let $\delta>0$ be arbitrary.
By Egorov's theorem there exists an $N_{\delta}\in\N$ and a set $\Xi\subset X$ of measure $>1-\delta$ such that for every $n\geq N_{\delta}$ and $\xi\in\Xi$ the average in \eqref{eq:BFKO-cond} is bounded by $\delta/2$.

Let $N_{\delta} \leq L \leq R$ be arbitrary and choose a continuous function $\eta : \R^{[L,R]} \to [0,1]$ that is $1$ when all its arguments are less than $\delta/2$ and $0$ when one of its arguments is greater than $\delta$ (here and later $[L,R]=\{L,L+1,\dots,R\}$).
Then by the Stone--Weierstrass theorem the function
\[
h(\xi):=\eta(|\E_{g\in F_{L}}f(gx)f(g\xi)|,\dots,|\E_{g\in F_{R}}f(gx)f(g\xi)|)
\]
lies in the closed convolution-invariant subalgebra of $L^{\infty}(X)$ spanned by $f$.

By the assumption $x$ is generic for $h$.
Since $h|_{\Xi} \equiv 1$, we have $\int_{X} h > 1-\delta$.
Hence the set of $a$ such that $h(ax)>0$ has lower density $>1-\delta$.

For every such $a$ we have
\[
| \E_{g\in F_{n}} f(gax)f(gx) | < \delta,
\quad L\leq n\leq R.
\qedhere
\]
\end{proof}
\begin{proof}[Proof of Lemma~\ref{lem:ae-perp}]
Let $f\in L^{\infty}(X)$ be orthogonal to the Kronecker factor.
By \cite[Theorem 1]{MR0174705} this implies that the ergodic averages of $f\otimes f$ converge to $0$ in $L^{2}(X\times X)$.
By the Lindenstrauss pointwise ergodic theorem \cite[Theorem 1.2]{MR1865397} this implies \eqref{eq:BFKO-cond} for a.e.\ $x\in X$.
Since a.e.\ $x\in X$ is also fully generic for $f$, the conclusion follows from Lemma~\ref{lem:to-zero-in-density}.
\end{proof}

\subsection{Lindenstrauss covering lemma}
Given a collection of intervals, the classical Vitali covering lemma allows one to select a disjoint subcollection that covers a fixed fraction of the union of the full collection.
The appropriate substitute in the setting of tempered F\o{}lner sequences is the Lindenstrauss random covering lemma.
It allows one to select a \emph{random} subcollection that is \emph{expected} to cover a fixed fraction of the union and to be \emph{almost} disjoint.
The almost disjointness means that the expectation of the counting function of the subcollection is uniformly bounded by a constant.
As such, the Vitali lemma is stronger whenever it applies, and the reader who is only interested in the standard F\o{}lner sequence in $\Z$ can skip this subsection.

We use two features of Lindenstrauss' proof of the random covering lemma that we emphasize in its formulation below.
The first feature is that the second moment (and in fact all moments) of the counting function is also uniformly bounded (this follows from the bound for the moments of a Poisson distribution).
The second feature is that the random covering depends measurably on the data.
We choose to include the explicit construction of the covering in the statement of the lemma instead of formalizing this measurability statement.
To free up symbols for subsequent use we replace the auxiliary parameter $\delta$ in Lindenstrauss' statement of the lemma by $C\inv$ and expand the definition of $\gamma$.

For completeness we recall that a \emph{Poisson point process} with intensity $\alpha$ on a measure space $(X,\mu)$ is a counting (i.e.\ atomic, with at most countably many atoms and masses of atoms in $\N$) measure-valued map $\Upsilon:\Omega \to M(X)$ such that for every finite measure set $A\subset X$ the random variable $\omega\mapsto \Upsilon(\omega)(A)$ is Poisson with mean $\alpha\mu(A)$ and for any disjoint sets $A_{i}$ the random variables $\omega\mapsto \Upsilon(\omega)|_{A_{i}}$ are jointly independent (here and later $\Upsilon|_{A}$ is the measure $\Upsilon|_{A}(B)=\Upsilon(A\cap B)$).
It is well-known that on every $\sigma$-finite measure space there exists a Poisson process.
\begin{lemma}[{\cite[Lemma 2.1]{MR1865397}}]
\label{lem:lindenstrauss-covering}
Let $G$ be a lcsc group with left Haar measure $|\cdot|$.
Let $(F_{N})_{N=L}^{R}$ be a $C$-tempered sequence.
Let $\Upsilon_{N}:\Omega_{N}\to M(G)$ be independent Poisson point processes with intensity $\alpha_{N}=\delta/|F_{N}|$ w.r.t.\ the right Haar measure $\rho$ on $G$ and let $\Omega:=\prod_{N}\Omega_{N}$.

Let $A_{N|R+1}\subset G$, $N=L,\dots,R$, be sets of finite measure.
Define (dependent!) counting measure-valued random variables $\Sigma_{N} : \Omega \to M(G)$ in descending order for $N=R,\dots,L$ by
\begin{enumerate}
\item $\Sigma_{N} := \Upsilon_{N}|_{A_{N|N+1}}$,
\item $A_{i|N} := A_{i|N+1}\setminus F_{i}\inv F_{N}\Sigma_{N} = \{ a\in A_{i|N+1} : F_{i}a\cap F_{N}\Sigma_{N}=\emptyset\}$ for $i<N$.
\end{enumerate}
Then for the counting function
\[
\Lambda = \sum_{N}\Lambda_{N},
\quad
\Lambda_{N}(g)(\omega) = \sum_{a\in\Sigma_{N}(\omega)}1_{F_{N}a}(g)
\]
the following holds.
\begin{enumerate}
\item $\Lambda$ is a measurable, a.s.\ finite function on $\Omega\times G$,
\item $\E(\Lambda(g)|\Lambda(g)\geq 1) \leq 1+C\inv$ for every $g\in F$,
\item $\E(\Lambda^{2}(g)|\Lambda(g)\geq 1) \leq (1+C\inv)^{2}$ for every $g\in F$,
\item $\E(\int\Lambda) \geq (2C)\inv |\cup_{N=L}^{R}A_{N|R+1}|$.
\end{enumerate}
\end{lemma}

\section{Orthogonality criterion for amenable groups}
In our view, the BFKO orthogonality criterion is a statement about bounded measurable functions on $G$.
We encapsulate it in the following lemma.
\begin{lemma}
\label{lem:orth}
Let $(F_{N})$ be a $C$-tempered strong F\o{}lner sequence.

Let $\epsilon>0$, $K\in\N$ and $\delta>0$ be sufficiently small depending on $\epsilon,K$.
Let $c\in L^{\infty}(G)$ be bounded by $1$ and $[L_{1},R_{1}],\dots,[L_{K},R_{K}]$ be a sequence of increasing intervals of natural numbers such that the following holds for any $j<k$ and any $N\in [L_{k},R_{k}]$.
\begin{enumerate}
\item $|\partial_{F_{(j)}}F_{N}|<\delta |F_{N}|$, where $F_{(j)}=\cup_{N=L_{j}}^{R_{j}} F_{N}$
\item $S_{\delta,L_{j},R_{j}}(c)$ has density at least $1-\delta$ in $F_{N}$.
\end{enumerate}
Let $f\in L^{\infty}(G)$ be bounded by $1$ and consider the sets
\[
A_{N}:=
\{a : |\E_{g\in F_{N}}c(g)f(ga)| \geq \epsilon\},
\quad
A_{(j)} := \cup_{N=L_{j}}^{R_{j}}A_{N}.
\]
Then for every compact set $I\subset G$ with $|I\cap F_{(j)}\inv I^{\complement}|<\delta |I|$ for every $j$ we have
\[
\frac1K \sum_{j=1}^{K} d_{I}(A_{(j)}) < \frac{5C}{\epsilon \sqrt{K}}.
\]
\end{lemma}
Under the assumption \eqref{eq:cond} a sequence $[L_{1},R_{1}],\dots,[L_{K},R_{K}]$ with the requested properties can be constructed for any $K$.
\begin{proof}
For $1\leq k\leq K$, $L_{k}\leq N \leq R_{k}$ let $\Upsilon_{N}:\Omega_{N}\to M(G)$ be independent Poisson point processes of intensity $\alpha_{N}=\delta |F_{N}|\inv$ w.r.t.\ the \emph{right} Haar measure.

Let $\Omega = \prod_{k=1}^{K}\prod_{N=L_{k}}^{R_{k}}\Omega_{N}$.
We construct random variables $\Sigma_{N} : \Omega \to M(A_{N})$ that are in turn used to define functions
\[
c^{(k)} := \sum_{N=L_{k}}^{R_{k}}\sum_{a\in\Sigma_{N}} \pm c|_{F_{N}}(\cdot a\inv),
\quad k=1,\dots,K,
\]
where the sign is chosen according to as to whether $\E_{g\in F_{N}}c(g)f(ga)$ is positive or negative.
These functions will be mutually nearly orthogonal on $I$ and correlate with $f$, from where the estimate will follow by a standard Hilbert space argument.

We construct the random variables in reverse order, beginning with $k=K$.
Let the set of ``admissible origins'' be
\[
O^{(j)}:=A_{(j)}\cap
\Big(\big(I
\setminus F_{(j)}\inv I^{\complement}
\big)
\setminus \cup_{k=j+1}^{K}\cup_{N=L_{k}}^{R_{k}}\cup_{a\in\Sigma_{N}} (\partial_{F_{(j)}}(F_{N}) \cup (S_{\delta,L_{j},R_{j}}(c)^{\complement} \cap F_{N}))a
\Big).
\]
This set consists of places where we could put copies of initial segments of $c$ in such a way that they would correlate with $f$ and would not correlate with the copies that were already used in the functions $c^{(k)}$ for $k>j$.

Let $A_{N|R_{j}+1} := O^{(j)}\cap A_{N}$ and construct random coverings $\Sigma_{N}$, $N=L_{j},\dots,R_{j}$ as in Lemma~\ref{lem:lindenstrauss-covering} (if the Vitali lemma is available then one can use deterministic coverings that it provides instead).
By Lemma~\ref{lem:lindenstrauss-covering} the counting function
\[
\Lambda^{(j)} = \sum_{N=L_{j}}^{R_{j}}\Lambda_{N},
\quad
\Lambda_{N}(g)(\omega) = \sum_{a\in\Sigma_{N}(\omega)}1_{F_{N}a}(g)
\]
satisfies
\begin{enumerate}
\item $\E(\Lambda^{(j)}(g)) \leq (1+C\inv)$ for every $g\in G$
\item $\E(\Lambda^{(j)}(g)^{2}) \leq (1+C\inv)^{2}$ for every $g\in G$
\item $\E(\int\Lambda^{(j)}) \geq (2C)\inv |O^{(j)}|$.
\end{enumerate}
In particular, the last condition implies that
\[
\E\int_{I} c^{(j)}f > \epsilon(2C)\inv |O^{(j)}|,
\]
while the second shows that $\|c^{(j)}\|_{L^{2}(\Omega\times I)}\leq (1+C\inv)|I|^{1/2}$.
Moreover, it follows from the definition of $O^{(j)}$ that
\[
|\E\int_{I} c^{(j)} c^{(k)} | \leq |I|\delta(1+C\inv)
\]
whenever $j<k$.
Using the fact that $|c^{(j)}|\leq \Lambda^{(j)}$ and the Hölder inequality we obtain
\begin{multline*}
\sum_{j=1}^{K} \epsilon(2C)\inv \E|O^{(j)}|
< \E\int_{I} \sum_{j=1}^{K} c^{(j)}f\\
\leq \big( \E\int_{I} \big( \sum_{j=1}^{K} c^{(j)} \big)^{2} \big)^{1/2} |I|^{1/2}
< |I| \sqrt{K(1+C\inv)^{2} + K^{2}\delta(1+C\inv)}.
\end{multline*}
This can be written as
\[
\frac1K \sum_{j=1}^{K} \E|O^{(j)}|
< \frac{2C|I|}{\epsilon} \sqrt{(1+C\inv)^{2}/K+\delta(1+C\inv)}.
\]
Finally, the set $O^{(j)}$ has measure at least
\begin{multline*}
|I|(d_{I}(A_{(j)})-\delta)
-\sum_{k=j+1}^{K}\sum_{N=L_{k}}^{R_{k}}\sum_{a\in\Sigma_{N}} (|\partial_{F_{(j)}}(F_{N}a)|+|(S_{\delta,L_{j},R_{j}}(c)^{\complement} \cap F_{N})a|)\\
\geq
|I| (d_{I}(A_{(j)})-\delta)
-2\delta \sum_{k=j+1}^{K}\sum_{N=L_{k}}^{R_{k}}\sum_{a\in\Sigma_{N}} |F_{N}a|,
\end{multline*}
(here we have used the largeness assumptions on $L_{k}$), so
\[
\E |O^{(j)}|
\geq
|I| (d_{I}(A_{(j)})-\delta)
-2\delta (K-j) |I| (1+C\inv)
>
|I| (d_{I}(A_{(j)})-4\delta K)
\]
and the conclusion follows provided that $\delta$ is sufficiently small.
\end{proof}

The BFKO criterion for measure-preserving systems follows by a transference argument.
\begin{proof}[Proof of Theorem~\ref{thm:good-weight-ptw-conv-to-zero}]
Assume that the conclusion fails for some measure-preserving system $(X,G)$ and $f\in L^{\infty}$.
Then we obtain some $\epsilon>0$ and a set of positive measure $\Xi\subset X$ such that
\[
\limsup_{N\to\infty} | \E_{g\in F_{N}}c(g)f(gx) | > 2\epsilon
\quad\text{for all } x\in\Xi.
\]
We may assume $\mu(\Xi)>\epsilon$.
Shrinking $\Xi$ slightly (so that $\mu(\Xi)>\epsilon$ still holds) we may assume that for every $\ul{N}\in\N$ there exists $F(\ul{N})\in\N$ (independent of $x$) such that for every $x\in\Xi$ there exists $\ul{N}\leq N\leq F(\ul{N})$ such that the above average is bounded below by $2\epsilon$.

Let $K > 25 C^{2} \epsilon^{-4}$ and $[L_{1},R_{1}],\dots,[L_{K},R_{K}]$ be as in Lemma~\ref{lem:orth} with $R_{j}=F(L_{j})$.
In this case that lemma says that at least one of the sets $\cup_{N=L_{j}}^{R_{j}}A_{N}$ has lower density less than $\epsilon$.

Choose continuous functions $\eta_{j} : \R^{[L_{j},R_{j}]}\to [0,1]$ that are $1$ when at least one of their arguments is greater than $2\epsilon$ and $0$ if all their arguments are less than $\epsilon$.
Let
\[
h(x) := \prod_{j=1}^{K} \eta_{j}(|\E_{g\in F_{L_{j}}}c(g)f(gx)|,\dots,|\E_{g\in F_{R_{j}}}c(g)f(gx)|).
\]
By construction of $F$ we know that $h|_{\Xi}\equiv 1$, so that $\int_{X} h > \epsilon$.
Let $x_{0}$ be a generic point for $h$ (e.g.\ any fully generic point for $f$), then $\ld\{a : h(ax_{0})>0\} > \epsilon$.
In other words,
\[
\ld\{ a : \forall j\leq K\,
\exists N \in [L_{j},R_{j}]\,
| \E_{g\in F_{N}} c(g)f(gax_{0}) | \geq \epsilon
\} > \epsilon.
\]
This contradicts Lemma~\ref{lem:orth} with $f(g)=f(gx_{0})$.
\end{proof}

For translations on compact groups we obtain the same conclusion everywhere.
It is not clear to us whether an analogous statement holds for general uniquely ergodic systems.
\begin{corollary}
\label{cor:good-weight-everywhere-conv-to-zero}
Let $G$ be a lcsc group with a $C$-tempered strong F\o{}lner sequence $(F_{n})$.
Let $c\in L^{\infty}(G)$ be a function bounded by $1$ that satisfies the condition \eqref{eq:cond}.
Let also $\Omega$ be a compact group and $\chi : G\to\Omega$ a continuous homomorphism.
Then for every $\phi\in C(\Omega)$ we have
\[
\lim_{N\to\infty} \E_{g\in F_{N}} c(g) \phi(\chi(g)\omega) = 0
\quad \text{for every } \omega\in\Omega.
\]
\end{corollary}
\begin{proof}
We may assume that $\chi$ has dense image, so that the translation action by $\chi$ becomes ergodic.
By Theorem~\ref{thm:good-weight-ptw-conv-to-zero} we obtain the conclusion a.e.\ and the claim follows by uniform continuity of $\phi$.
\end{proof}

\section{Return times theorem}
We turn to the deduction of the return times theorem (Theorem~\ref{thm:rtt-amenable}).
This will require two distinct applications of Theorem~\ref{thm:good-weight-ptw-conv-to-zero}.
We begin with a Wiener--Wintner type result.

Recall that the Kronecker factor of a measure-preserving dynamical system corresponds to the reversible part of the Jacobs--de Leeuw--Glicksberg decomposition of the associated Koopman representation.
In particular, it is spanned by the finite-dimensional $G$-invariant subspaces of $L^{2}(X)$.
We refer to \cite{isem-book} for a treatment of the JdLG decomposition.

Let $F \subset L^{2}(X)$ be a $d$-dimensional $G$-invariant subspace and $f\in F$.
We will show that for a.e.\ $x\in X$ we have $f(gx) = \phi(\chi(g)u)$ for some $\phi\in C(U(d))$, continuous representation $\chi : G \to U(d)$, and a.e.\ $g\in G$.
To this end choose an orthonormal basis $(f_{i})_{i=1,\dots,d}$ of $F$.
Then by the invariance assumption we have $f_{i}(g\cdot) = \sum_{j} c_{i,j} f_{j}(\cdot)$, and the matrix $(c_{i,j})=:\chi(g)$ is unitary since the $G$-action on $X$ is measure-preserving.
This gives us a $d$-dimensional measurable representation $\chi$ that is automatically continuous \cite[Theorem 22.18]{0837.43002}.
The point $u=(u_{i})$ is given by the coordinate representation $f=\sum u_{i}f_{i}$.
Thus we have $f(g\cdot) = \sum_{i}(\chi(g)u)_{i} f_{i}(\cdot)$ in $L^{2}(X)$ and hence, fixing some measurable representatives for $f_{i}$'s, a.e.\ on $X$.
By Fubini's theorem we obtain a full measure subset of $X$ such that the above identity holds for a.e.\ $g\in G$.
For every $x$ from this set we obtain the claim with the continuous function $\phi(U) = \sum_{i} (Uu)_{i}f_{i}(x)$.

\begin{corollary}[Wiener--Wintner theorem]
\label{cor:wiener-wintner}
Let $G$ be a lcsc group with a $C$-tempered strong F\o{}lner sequence $(F_{n})$.
Then for every ergodic measure-preserving system $(X,G)$ and every $f\in L^{\infty}(X)$ there exists a full measure set $\tilde X \subset X$ such that the following holds.
Let $\Omega$ be a compact group and $\chi : G\to\Omega$ a continuous homomorphism.
Then for every $\phi\in C(\Omega)$, every $\omega\in\Omega$ and every $x\in \tilde X$ the limit
\[
\lim_{N\to\infty} \E_{g\in F_{N}} f(gx) \phi(\chi(g)\omega)
\]
exists.
\end{corollary}
\begin{proof}
By Lemma~\ref{lem:ae-perp} and Corollary~\ref{cor:good-weight-everywhere-conv-to-zero} we obtain the conclusion for $f$ orthogonal to the Kronecker factor.

By linearity and in view of the Lindenstrauss maximal inequality \cite[Theorem 3.2]{MR1865397} it remains to consider $f$ in a finite-dimensional invariant subspace of $L^{2}(X)$.
In this case, for a.e.\ $x\in X$ we have $f(gx)=\phi'(\chi'(g)u_{0})$ for some finite-dimensional representation $\chi':G\to U(d)$, some $u_{0}\in U(d)$, some $\phi'\in C(U(d))$ and a.e.\ $g\in G$.
The result now follows from uniqueness of the Haar measure on the closure of $\chi\times\chi'(G)$.
\end{proof}
A different proof using unique ergodicity of an ergodic group extension of a uniquely ergodic system can be found in \cite{MR1195256}.

Finally, the return times theorem follows from a juxtaposition of previous results.
\begin{proof}[Proof of Theorem~\ref{thm:rtt-amenable}]
By Lemma~\ref{lem:ae-perp} and Theorem~\ref{thm:good-weight-ptw-conv-to-zero} the conclusion holds for $f\in L^{\infty}(X)$ orthogonal to the Kronecker factor.

By linearity and in view of the Lindenstrauss maximal inequality \cite[Theorem 3.2]{MR1865397} it remains to consider $f$ in a finite-dimensional invariant subspace of $L^{2}(X)$.
In this case, for a.e.\ $x\in X$ we have $f(gx)=\phi(\chi(g)u_{0})$ for some finite-dimensional representation $\chi:G\to U(d)$, some $u_{0}\in U(d)$, some $\phi\in C(U(d))$ and a.e.\ $g\in G$.
The conclusion now follows from Corollary~\ref{cor:wiener-wintner}.
\end{proof}

\printbibliography
\end{document}